\documentclass[10pt]{amsart}
\usepackage{enumerate,amsmath,amssymb,latexsym,
amsfonts, amsthm, amscd, MnSymbol}

\theoremstyle{plain}

\newtheorem{theorem}{Theorem}

\numberwithin{equation}{section}

\newcommand{\R}{\mathbb{R}}
\newcommand{\CC}{\mathbb{C}}

\newcommand{\ssl}{{\rm sl}}
\newcommand{\ccl}{{\rm cl}}
\newcommand{\ii}{{\rm i}}

\addtocounter{section}{0}

\begin{document}

\title {The Lemniscatic Functions}

\date{}

\author[P.L. Robinson]{P.L. Robinson}

\address{Department of Mathematics \\ University of Florida \\ Gainesville FL 32611  USA }

\email[]{paulr@ufl.edu}

 \keywords{}

\begin{abstract}

We develop the theory of the lemniscatic functions $\ssl$ and $\ccl$ from their definition as solutions to an initial value problem. 

\end{abstract}

\maketitle

\medbreak

\section{Introduction} 

\medbreak 

The lemniscatic functions $\ssl$ (or {\it sinus lemniscaticus}) and $\ccl$ (or {\it cosinus lemniscaticus}) are frequently introduced as special cases of the Jacobian elliptic functions: thus, 
$$\sqrt{2} \: \ssl (z) = {\rm sd} (\sqrt{2} \: z) \; \; {\rm and} \; \; \ccl (z) = {\rm cn} (\sqrt{2} \: z)$$
where the indicated functions ${\rm sd} = {\rm sn}/{\rm dn}$ and ${\rm cn}$ have self-complementary modulus $1/ \sqrt{2}$; they are also frequently introduced by inversion of the corresponding Abelian integrals. 

\medbreak 

Our purpose here is to develop these functions {\it ab initio}, not by the inversion of integrals but rather directly, as the solutions to initial value problems. In addition to providing an independent construction of these functions, our account serves to illustrate the effectiveness for these initial value problems of the Picard existence-uniqueness theorem and the principle of analytic continuation, in conjunction with Weierstrassian duplication and the Schwarz Reflexion Principle. 

\medbreak 

\section{The Lemniscatic Functions}

\medbreak 

The trigonometric functions $\sin$ and $\cos$ may be defined as solutions to the initial value problem 
$$f\,' = g, \; g\,' = - f; \; \; f(0) = 0, \; g(0) = 1.$$
When we view this as a complex system, its solutions extend throughout $\CC$ as singly-periodic entire functions. 

\medbreak 

Here, we introduce the lemniscatic functions $\ssl$ and $\ccl$ as solutions to the related  initial value problem (signified by {\bf IVP} in what follows) 
$$s\,' = c \, (1 + s^2), \; c\,' = - s \, (1 + c^2); \; \; s(0) = 0, \; c(0) = 1.$$
As we shall see, the solutions to this complex system extend throughout $\CC$ as doubly-periodic meromorphic functions. 

\medbreak 

To start the construction of this pair of functions, it is convenient to begin with the function $s$ alone. In order to motivate the appropriate initial value problem for $s$ we establish at once the following counterpart to the `Pythagorean' identity. 

\medbreak 

\begin{theorem} \label{Py}
Solutions $s$ and $c$ to {\bf IVP} on a connected open set containing $0$ satisfy 
$$s^2 + s^2 c^2 + c^2 = 1.$$
\end{theorem} 

\begin{proof} 
This fundamental identity is an immediate consequence of {\bf IVP}: by straightforward differentiation, 
$$(s^2 + s^2 c^2 + c^2)\,' = 2 s s\,' (1 + c^2) + (1 + s^2) 2 c c\,' = 0$$
so that $s^2 + s^2 c^2 + c^2$ is constant, its value at $0$ holding throughout its domain. 
\end{proof} 

\medbreak 

We shall see many consequences to this formula, one of whose manifestations is the equivalent form 
$$(1 + s^2)(1 + c^2) = 2.$$

\medbreak 

When Theorem \ref{Py} is taken into account, the differential equation for $s$ squares to yield 
$$(s\,')^2 = c^2 (1 + s^2)(1 + s^2) = (1 - s^2)(1 + s^2) = 1 - s^4.$$
For this differential equation, the initial condition $s(0) = 0$ alone entails $s\,' (0) = \pm 1$; its companion $c(0) = 1$ singles out $s\,'(0) = 1$. Thus, $s$ satisfies the first-order initial value problem 
$$s\,' = (1 - s^4)^{1/2}; \; s(0) = 0$$ 
wherein the square-root is initially determined by its principal value. 

\medbreak 

\begin{theorem} \label{s}
The initial value problem 
$$s\,' = (1 - s^4)^{1/2}; \; s(0) = 0$$
has a unique solution in the open disc $B_r(0)$ of radius $r = 2^{-1/2}$ about $0$. 
\end{theorem} 

\begin{proof} 
An application of the Picard existence-uniqueness theorem, of which Theorem 2.3.1 in [1] is a suitable version. Fix $0 < R < 1$: if $|w| < R$ then $|(1 - w^4)^{1/2}| < (1 + R^4)^{1/2} = : M$; as the necessary Lipschitz condition is satisfied, the stated initial value problem for $s$ has a unique solution in the open disc $B_{R/M} (0)$. Here, we let $R \uparrow 1$ so that $R/M \uparrow r : =  2^{-1/2}$ to complete the proof. 
\end{proof} 

\medbreak 

Alternatively, this initial value problem may be presented in the essentially equivalent form
$$(s\,')^2 = 1 - s^4; \; \; s\,'(0) = 1.$$ 

\medbreak 

Note that here, the solution $s$ maps the open disc $B_r (0)$ to the open unit disc $B_1(0)$ by its very construction; in particular, $s$ as defined takes neither $\ii$ nor $-\ii$ as a value, so $1 + s^2$ is nowhere zero. Further, note from $(s\,')^2 = 1 - s^4$ by differentiation that $2 s\,' s\,'' = - 4 s^3 s\,'$ whence the holomorphic function $s$ solves the second-order initial value problem 
$$s\,'' = - 2 s^3; \; \; s(0) = 0, \; s\,'(0) = 1.$$ 

\medbreak 

Incidentally, although we have here chosen to present $s$ as the solution to a first-order initial value problem, there will be points at which this approach is inadequate, because the Lipschitz condition for the Picard theorem is not met. Second-order initial value problems for the differential equation `$s\,'' = - 2 s^3$' are always amenable to the Picard existence-uniqueness theorem, as $- 2 s^3$ is polynomial in $s$. 

\medbreak 

We may now confirm the unique existence of a solution to {\bf IVP} in the open disc $B_r (0)$.  

\medbreak 

\begin{theorem} \label{IVP}
The initial value problem 
$$s\,' = c \, (1 + s^2), \; c\,' = - s \, (1 + c^2); \; \; s(0) = 0, \; c(0) = 1$$ 
has a unique solution in the disc $B_r (0)$. 
\end{theorem} 

\begin{proof} 
We take the function $s$ guaranteed by Theorem \ref{s} and define the holomorphic function $c$ in $B_r (0)$ by 
$$c = s\,' / (1 + s^2).$$
The initial conditions $s(0) = 0$ and $c(0) = 1$ are evident from the definitions, as is the differential equation $s\,' = c \, (1 + s^2)$; for the companion differential equation $c\,' = - s \, (1 + c^2)$, we calculate 
$$(1 + s^2)^2 \, c\,' = (1 + s^2) s\,'' - s\,' 2 s s\,' = (1 + s^2) (- 2 s^3) - 2 s (1 - s^4)$$
whereupon cancellation results in  
$$(1 + s^2) \, c\,' = - 2s(s^2 + 1 - s^2) = - 2 s$$
and it remains only to invoke the remark following Theorem \ref{Py}. This proves existence; uniqueness is clear in view of Theorem \ref{s}. 
\end{proof} 

\medbreak 

It is readily verified that $c$ solves the companion second-order initial value problem 
$$c\,'' = - 2 c^3; \; \; c(0) = 1, \; c\,'(0) = 0.$$ 

\medbreak 

The lemniscatic functions admit a number of symmetries that are conveniently handled for $s$ and $c$ in tandem. We shall establish these symmetries in the open disc $B_r (0)$; once the functions are extended to larger appropriately symmetric connected domains, the symmetries will also extend (and the extensions will continue to satisfy {\bf IVP}) on account of the Identity Theorem. 

\medbreak 

\begin{theorem} \label{real} 
If $z \in B_r (0)$ then $s (\overline{z}) = \overline{s(z)}$ and $c (\overline{z}) = \overline{c(z)}$. 
\end{theorem} 

\begin{proof} 
Define functions $S$ and $C$ in $B_r (0)$ by the rules $S(z) = \overline{s(\overline{z})}$ and $C(z) = \overline{c(\overline{z})}$. It is readily verified that both 
$$S\,' = C (1 + S^2), \; \; C\,' = - S (1 + C^2)$$
and 
$$S(0) = 0, \; \; C(0) = 1.$$ 
By the uniqueness clause in Theorem \ref{IVP} we conclude that $(S, C) = (s, c)$. 
\end{proof} 

\medbreak

In other words, we may speak of the functions $s$ and $c$ as being `real': in particular, they are real-valued on the interval $( - r, r) = \R \cap B_r (0)$ and the coefficients in their Taylor expansions about $0$ are real. 

\medbreak 

As regards parity, $s$ is odd and $c$ is even. 

\medbreak 

\begin{theorem} \label{parity} 
If $z \in B_r (0)$ then $s(- z) = - s(z)$ and $c(- z) = c(z)$. 
\end{theorem} 

\begin{proof} 
Define functions $S$ and $C$ in $B_r (0)$ by the rules $S(z) = - s(-z)$ and $C(z) = c(-z)$. By differentiation and evaluation, the pair $(S, C)$ satisfies {\bf IVP}; so Theorem \ref{IVP} justifies the identifications $S = s$ and $C = c$. 
\end{proof} 

\medbreak 

The following symmetry essentially captures the `imaginary transformation' of Jacobi in the present context. 

\medbreak 

\begin{theorem} \label{i}
If $z \in B_r(0)$ then $s(\ii z) = \ii s(z)$ and $c(\ii z) = 1/c(z)$. 
\end{theorem} 

\begin{proof} 
Now routine: define $S$ and $C$ in $B_r (0)$ by the rules $S(z) = - \ii s(\ii z)$ and $C(z) = 1/c(\ii z)$; then verify that the pair $(S, C)$ satisfies {\bf IVP} and invoke Theorem \ref{IVP}. 
\end{proof} 

\medbreak 

Perhaps we ought to remark here that the identity in Theorem \ref{Py} and the fact that $|s| < 1$ throughout $B_r (0)$ together ensure that $c$ vanishes nowhere in $B_r (0)$. Of course, Theorem \ref{i} has Theorem \ref{parity} as an immediate corollary. 

\medbreak 

From this point on, the construction of the lemniscatic functions proceeds in two stages. The first stage is to extend the functions from $B_r (0)$ to the disc $B_{2 r} (0)$ of radius $2 r = \sqrt2$ by duplication formulae, and to reduplicate if need be: appropriate duplication formulae for $s$ and $c$ jointly are 
$$s(2 z) = \frac{2 s(z) c(z)}{1 - s(z)^2 c(z)^2}\: ,$$
$$c(2 z) = \frac{c(z)^2 - s(z)^2}{1 + s(z)^2 c(z)^2}\: ;$$
naturally, the justification of these formulae is part of the process. The second stage is to complete the extension of $s$ and $c$ as meromorphic functions in the plane by handing the matter over to the Schwarz Reflexion Principle, once duplication arrives at configurations to which reflexion may be conveniently applied. 

\medbreak 

It is certainly possible to extend $s$ and $c$ beyond the disc $B_r (0)$ jointly, as solutions of {\bf IVP}. However, at this point we prefer to proceed further with the two functions separately: first extending the lemniscatic sine; then extending the lemniscatic cosine. 

\medbreak 

\section{The Lemniscatic Sine} 

\medbreak 

In this Section, we shall extend the lemniscatic sine function $s$ beyond its current domain; in fact, we shall extend the holomorphic function $s$ from $B_r(0)$ to the entire plane as an elliptic function, whose characteristic features we shall identify. 

\medbreak 

As mentioned at the close of the preceding Section, the first phase of the extension process is effected by duplication. This requires a formula that expresses the value of $s$ at $2 z$ in terms of the values of $s$ and its derivative at $z$. We shall make no attempt at discovering such a formula; we shall merely accept it for use. Needless to say, the formula that we use will be rigorously justified as we go. 

\medbreak 

Explicitly, we define the function $S$ in $B_{2 r} (0)$ by the rule that if $z \in B_r (0)$ then 
$$S(2 z) = \frac{2 s(z) s\,'(z)}{1 + s(z)^4}\, .$$
Observe that $S$ is holomorphic, because $s$ is holomorphic and $|s| < 1$ throughout $B_r (0)$. It is a straightforward (though moderately tedious) exercise to verify that 
$$(S\,')^2 = 1 - S^4$$
on account of the identities $(s\,')^2 = 1 - s^4$ and $s\,'' = - 2 s^3$; also, $S(0) = 0$ and $S\,'(0) = 1$. Theorem \ref{s} now assures us that the restriction of $S$ to $B_r(0)$ is $s$ precisely; otherwise said, $S$ is a holomorphic extension of $s$ to the open disc $B_{2 r} (0)$. This being the case, the capitalization has served its clarifying purpose and will be dropped: from now on, we shall write simply $s : B_{2 r} (0) \to \CC$ for the extended function. The formula by which the extension was defined then becomes a duplication formula: if $z \in B_r(0)$ then 
$$s(2 z) = \frac{2 s(z) s\,'(z)}{1 + s(z)^4}\, .$$ 
Moreover, the symmetries established for $s$ on $B_r(0)$ in Theorem \ref{real}, Theorem \ref{parity} and Theorem \ref{i} continue to hold for $s$ on $B_{2 r}(0)$. We may take the liberty of using the same names for these theorems in this extended context.

\medbreak 

We now entertain the possibility of further extension to $B_{4 r} (0)$ by reduplication. Here we encounter a potential difficulty: whereas $1 + s^4$ is nowhere zero in the disc $B_r (0)$, it may have zeros in the disc $B_{2 r} (0)$; indeed there it does have zeros, which double up to poles in $B_{4 r}(0)$. Our next task is to locate these problematic points and be sure to account for all of them. 

\medbreak 

The following results, regarding the behaviour of $s$ on the real and imaginary diameters of discs and their angle bisectors, will be useful. 

\medbreak 

\begin{theorem} \label{X}
Let $z \in B_{2 r} (0)$: if $z^4 \in \R$ then $s(z)$ is a real multiple of $z$.  
\end{theorem} 

\begin{proof} 
Let $z \in B_{2 r} (0)$. From Theorem \ref{real} (extended) we deduce that if $z$ is real then so is $s(z)$; from Theorem \ref{i} (extended) we further deduce that if $z$ is purely imaginary then so is $s(z)$. Note that 
$$\{t \pm \ii t : t \in \R \} = \{ w \in \CC : \overline{w} = \mp \ii w\}.$$
It follows that if $z \in \{t + \ii t : t \in \R \}$ then $\overline{z} = - \ii z$ whence Theorem \ref{real} and Theorem \ref{i} yield  
$$\overline{s(z)} = s(\overline{z}) = s(- \ii z) = - \ii s(z)$$
and therefore $s(z) \in \{t + \ii t : t \in \R \}$. Similarly for $\{t - \ii t : t \in \R \}$.
\end{proof} 

\medbreak 

Consider further the behaviour of $s$ on the interval $(- r, r) = \R \cap B_r (0)$. From the fact that $|s| < 1$ on the disc $B_r(0)$ we deduce that $s\,' = (1 - s^4)^{1/2}$ where the positive square-root is taken throughout the interval $(-r, r)$; in particular, $s$ strictly increases there. The behaviour of $s$ on $\ii \R \cap B_r (0)$ matches this, by Theorem \ref{i}.  

\medbreak 

For convenience, let us agree to write 
$$\gamma = \frac{1 + \ii}{\sqrt{2}} = e^{\pi \ii/4}$$
for the square-root of $\ii$ in the positive quadrant. According to Theorem \ref{X}, $s$ maps the portion of $B_{2 r} (0)$ that lies in the line $\{t + \ii t : t \in \R \} = \gamma \R$ to the same line. The function $f : B_{2 r} (0) \to \CC$ defined by the rule that if $z \in B_{2 r} (0)$ then 
$$s(\gamma z) = \gamma f(z)$$
is thus real-valued on the interval $(- 2 r, 2 r)$ and there satisfies the differential equation 
$$f\,' = (1 + f^4)^{1/2}$$
wherein the square-root is positive (indeed, at least unity) throughout the interval. 

\medbreak 

We may now improve upon Theorem \ref{X} on $B_r(0)$ as follows. 

\medbreak 

\begin{theorem} \label{X+}
Let $z \in B_r (0)$: if $z^4 \in \R$ then $s(z)$ is a positive multiple of $z$. 
\end{theorem} 

\begin{proof} 
Let $z \in B_r(0)$. As $s$ is strictly increasing on $(-r, r)$ and vanishes at zero, the claim holds when $z$ is real; it holds when $z$ is purely imaginary by Theorem \ref{i}. The function $f : (- 2 r, 2 r) \to \R$ considered immediately prior to the present theorem has strictly positive derivative and is therefore strictly increasing; in particular, if $t \in (- r, r)$ then $f(t)$ is a positive multiple of $t$ and $s(\gamma t)$ is a positive multiple of $\gamma t$. This covers behaviour on $\gamma \R$; behaviour on $\overline{\gamma} \R$ follows similarly or by Theorem \ref{real}. 
\end{proof} 

\medbreak 

Note from the proof that $s$ is injective along the real and imaginary diameters and their angle bisectors. 

\medbreak 

The following result often serves as a foundation from which to develop the lemniscatic sine. 

\medbreak 

\begin{theorem} \label{integral} 
If $z \in B_r (0)$ then 
$$z = \int_0^{s(z)} \frac{{\rm d} \sigma \; \; \; }{(1 - \sigma^4)^{1/2}}\, .$$
\end{theorem} 

\begin{proof} 
The function $s$ maps $B_r (0)$ to the open unit disc, in which the function $\sigma \mapsto (1 - \sigma^4)^{-1/2}$ is holomorphic. By the chain rule, the holomorphic composite  
$$B_r (0) \to \CC : z \mapsto \int_0^{s(z)} \frac{{\rm d} \sigma \; \; \; }{(1 - \sigma^4)^{1/2}}$$
has derivative constantly one; and the composite vanishes at zero. 
\end{proof} 

\medbreak  

As a consequence, we have a converse to Theorem \ref{X} on $B_r (0)$. 

\medbreak 

\begin{theorem} \label{XX}
Let $z \in B_r(0)$: if $s(z)^4 \in \R$ then $z^4 \in \R$. 
\end{theorem} 

\begin{proof} 
Fix $z \in B_r (0)$ such that $s(z)^4 \in \R$; note that $s(z)^4 \in (- 1, 1)$ in fact. Evaluate the integral in Theorem \ref{integral} along the line segment $[0, s(z)]$: with $\sigma = s(z) t$ we have 
$$z = \int_0^1 \frac{s(z){\rm d} t \; \; \; }{(1 - s(z)^4 t^4)^{1/2}} = \Big[\int_0^1 \frac{{\rm d} t \; \; \; }{(1 - s(z)^4 t^4)^{1/2}}\Big] \: s(z)$$
where the integral within $\big[ \dots \big]$ is a strictly positive real number. 
\end{proof} 

\medbreak 

Recall the function $f : (- 2 r, 2 r) \to \R$ defined ahead of Theorem \ref{X+}. 

\medbreak 

\begin{theorem} \label{K}
There exists a unique $K \in (0, 1)$ such that $f(K) = 1$. 
\end{theorem} 

\begin{proof} 
The function $f$ on the interval $(- 2 r, 2 r) = (- \sqrt{2}, \sqrt{2})$ has derivative $f\,' = (1 + f^4)^{1/2}$ with value at least unity; as $f > 0$ on $(0, \sqrt{2})$ the intermediate value theorem concludes the argument. 
\end{proof} 

\medbreak 

From its definition, $K$ may be expressed as a definite integral, thus 
$$K = \int_0^1 \frac{{\rm d} \tau \; \; \; }{(1 + \tau^4)^{1/2}} = 0.92703733865... \: .$$

\medbreak 

It now follows that $s(\pm K \gamma) = \pm \gamma$ and $s(\pm K \overline{\gamma}) = \pm \overline{\gamma}$ : in particular, $s^4 = -1$ at each of the four points $\pm K \gamma$ and $\pm K \overline{\gamma}$ in $B_{2 r} (0)$. 

\medbreak 

We are now in a position to confirm that we have located all the points in the disc $B_{2 r} (0)$ at which $s^4 = -1$. 

\medbreak 

\begin{theorem} \label{s4}
The points in $B_{2 r} (0)$ at which $s^4 = -1$ are precisely $\pm K \gamma$ and $\pm K \overline{\gamma}$. 
\end{theorem} 

\begin{proof} 
We have seen that $s^4 = -1$ at each of the four listed points. Let $z \in B_{2 r} (0)$ satisfy $s(z)^4 = -1$. Write $z = 2 w$ with $w \in B_r (0)$. From the duplication formula for $s$ and the formula $(s\,')^2 = 1 - s^4$ we deduce that   
$$- 1 = s(2 w)^4 = \Big( \frac{2 s(w) s\,'(w)}{1 + s(w)^4} \Big)^4 = 16 \frac{s(w)^4 (1 - s(w)^4)^2}{(1 + s(w)^4)^4}$$
and therefore that $s(w)^4$ is a zero of the quartic  
$$q(\sigma) = (\sigma + 1)^4 + 16 \sigma (\sigma- 1)^2 , $$
which quartic factors thus 
$$q(\sigma) = (\sigma^2 + 10 \sigma + 1)^2 - 128 \sigma^2 = (\sigma^2 + (10 - 8 \sqrt{2}) \sigma + 1)(\sigma^2 + (10 + 8 \sqrt{2}) \sigma + 1).$$ 
Here, the first quadratic factor has both of its roots on the unit circle, while the second has negative real roots that straddle the unit circle. Thus, $s(w)^4 \in B_1(0)$  can have only one value: explicitly, $2 \sqrt{14 + 10 \sqrt{2}} - 5 - 4 \sqrt{2}$. Theorem \ref{XX} now places $s(w)$ on one of the four radii of the disc $B_r (0)$ that pass through the points $\pm \gamma$ and $\pm \overline{\gamma}$. As we noted after Theorem \ref{X+} that $s$ is injective on each of these radii, we deduce that $w$ (hence $z$ also) can have only four values. We conclude that our list of points in $B_{2 r} (0)$ at which $s^4 = -1$ is complete. 
\end{proof} 

\medbreak

At last we may now return to the task of reduplication, with the following result. 

\medbreak 

\begin{theorem} \label{4K} 
The holomorphic function $s$ on $B_r (0)$ extends to a function that is holomorphic in the open disc $B_{4 r} (0)$ except for simple poles at the four points $\pm 2 K \gamma$ and $\pm 2 K \overline{\gamma}$. 
\end{theorem} 

\begin{proof} 
Reduplication. With $s : B_{2 r} (0) \to \CC$ as defined at the start of this section, define $S$ in $B_{4 r} (0) \setminus \{ \pm 2 K \gamma, \: \pm 2 K \overline{\gamma}\}$ by declaring that if $z \in B_{2 r} (0) \setminus \{ \pm K \gamma, \: \pm K \overline{\gamma}\}$ then 
$$S(2 z) = \frac{2 s(z) s\,'(z)}{1 + s(z)^4}\, .$$
After Theorem \ref{s4} we know that $S$ is holomorphic. Direct calculation shows that $(S\,')^2 = 1 - S^4$ along with $S(0) = 0$ and $S\,'(0) = 1$. It therefore follows from Theorem \ref{s} that the restriction of $S$ to $B_r (0)$ is $s$. Finally, the singularities of $S$ at the four indicated points are indeed simple poles, as is readily checked.  

\end{proof} 

\medbreak 

As usual, we shall continue the notation $s$ for the extended function. Notice that $B_{2 K} (0)$ is the largest disc about $0$ on which the function $s$ is holomorphic. 

\medbreak 

\begin{theorem} \label{L}
There exists a unique $L \in [0, \tfrac{1}{2} \pi]$ such that $s(L) = 1$. 
\end{theorem} 

\begin{proof} 
Let $t$ in $(0, 4 r)$ be such that $s(u) < 1$ whenever $0 < u < t$: by continuation of Theorem \ref{integral} it follows that 
$$t = \int_0^{s(t)}  \frac{{\rm d} \sigma \; \; \; }{(1 - \sigma^4)^{1/2}} < \int_0^{s(t)}  \frac{{\rm d} \sigma \; \; \; }{(1 - \sigma^2)^{1/2}} <  \frac{1}{2} \pi;$$ 
the set of all such $t$ thus has $\frac{1}{2} \pi$ as an upper bound, so by continuity its supremum $L \leqslant \frac{1}{2} \pi$ is a point at which $s$ takes the value $1$. This settles existence; uniqueness is an exercise (or see below). 
\end{proof} 

\medbreak 

From its definition, $L$ may be expressed as a definite integral, thus 
$$L = \int_0^1 \frac{{\rm d} \sigma \; \; \; }{(1 - \sigma^4)^{1/2}} = 1.311028777146... \: .$$

\medbreak 

As a matter of fact, $K$ and $L$ are related quite simply: thus 
$$L = \sqrt2 K.$$ 
In the form 
$$\int_0^1 \frac{{\rm d} \sigma \; \; \; }{(1 - \sigma^4)^{1/2}} = \sqrt2 \int_0^1 \frac{{\rm d} \tau \; \; \; }{(1 + \tau^4)^{1/2}}$$
this may be verified by the substitution 
$$\sigma^2 = \frac{2 \tau^2}{1 + \tau^4}.$$
We may instead deduce this relationship from the following result, which is a special case of the addition formula for the lemniscatic sine and essentially amounts to a further examination of the function $f$ that was introduced prior to Theorem \ref{X+}; as usual, the displayed identity will continue to hold (except on a corresponding discrete set) when $s$ is meromorphically extended. 

\medbreak 

\begin{theorem} \label{KL}
The identity 
$$s ( \sqrt2 \, \overline{\gamma} \, z)^2 = - 2 \ii \, \frac{s(z)^2}{1 - s(z)^4}$$ 
holds whenever $z \in B_{4 r} (0)$ is not one of the eight points 
$$\pm 2 K \gamma, \: \pm 2 K \overline{\gamma}, \: \pm L, \: \pm L \ii.$$
\end{theorem} 

\begin{proof} 
Define $S : B_1(0) \to \CC$ by the rule 
$$S(z) = \sqrt2 \, \overline{\gamma} \, \frac{s(\frac{\gamma}{\sqrt2} z)}{(1 - s(\frac{\gamma}{\sqrt2} z)^4)^{1/2}}.$$
Here, the principal square-root is taken and $S$ is holomorphic because $|s| < 1$ on $B_r (0)$. By direct calculation, $S$ satisfies the initial value problem `$S\,' = (1 - S^4)^{1/2}, \; S(0) = 0$' and so agrees with $s$ on $B_r (0)$ by Theorem \ref{s}. After squaring and a change of variable, the Identity Theorem ensures that the relation 
$$(1 - s(z)^4) \, s ( \sqrt2 \, \overline{\gamma} \, z)^2 = - 2 \ii s(z)^2$$
holds for all $z$ in the (connected) common domain of $s$ and the function $s(\sqrt2 \, \overline{\gamma} \, \bullet)$. As developed thus far, this common domain is $B_{4 r} (0)$ less the eight listed points.  
\end{proof} 

\medbreak 

Evaluation of the identity in Theorem \ref{KL} at the point $z = K \gamma$ recovers the identity 
$$L = \sqrt2 \, K.$$
In spite of this identity we shall retain both $K$ and $L$, as each leads to simplifications: as we have seen, $2 K$ is the radius of the largest disc about $0$ on which $s$ is holomorphic; as we shall see, the fully extended $s$ has $4 L$ as a period. 

\medbreak 

We now prepare to hand over the continued extension of $s$ to the Schwarz Reflexion Principle. Let us write $U$ for the open lune with axis the interval $(2 K \overline{\gamma}, 2 K \gamma)$ centred at $L$ and having the real points $2 K$ and $2 L - 2 K$ on its boundary. 

\medbreak 

\begin{theorem} \label{lune}
If $z$ lies in the open lune $U$ then $s( 2 L - z) = s(z)$. 
\end{theorem} 

\begin{proof} 
Note that the map $z \mapsto 2 L - z$ leaves invariant the lune $U$ and fixes its centre $L$. Define $S : U \to U$ by the rule $S(z) = s(2 L - z)$: then $S\,' (z) = - s\,'(2 L - z)$ and 
$$S\,'' (z) = s\,''(2 L - z) = - 2 s(2 L - z)^3 = - 2 S(z)^3$$
along with $S(L) = s(L) = 1$ and $S\,'(L) = - s\,' (L) = 0$. Picard gives the second-order initial value problem `$s\,'' = - 2 s^3; \; s(L) = 1, \, s\,'(L) = 0$' a unique solution near $L$; an application of the Identity Theorem ends the proof. 
\end{proof} 

\medbreak 

Of course, if we knew that $s$ extends meromorphically to the plane, this result (in conjunction with the Identity Theorem) would imply that the meromorphic extension has $4 L$ as a period, thus: 
$$s(4 L + z) = s(2 L - (z - 2 L)) = s(z - 2 L) = - s(2 L - z) = - s(-z) = s(z).$$

\medbreak 

\begin{theorem} \label{edge}
The function $s$ is real-valued along the interval $(2 K \overline{\gamma}, 2 K \gamma)$. 
\end{theorem} 

\begin{proof} 
If $z \in (2 K \overline{\gamma}, 2 K \gamma)$ then $z$ has real part $L$: thus $\overline{z} = 2 L - z$ and so 
$$\overline{s(z)} = s(\overline{z}) = s (2 L - z) = s(z)$$ 
on account of Theorem \ref{real} and Theorem \ref{lune}. 
\end{proof} 

\medbreak 

Now Theorem \ref{parity} implies that $s$ is real-valued along $(- 2 K \overline{\gamma}, - 2 K \gamma)$, while Theorem \ref{i} implies that $s$ has purely imaginary values along $(- 2 K \overline{\gamma}, 2 K \gamma)$ and along $(- 2 K \gamma, 2 K \overline{\gamma})$. This means that we may apply the Schwarz Reflexion Principle to the function $s$ on the square with vertices $\pm 2 K \overline{\gamma}$ and $\pm 2 K \gamma$.  

\medbreak 

We shall refrain from presenting the full details of this application of the Schwarz Reflexion Principle; however, it is appropriate to mention one or two aspects of the present situation. 

\medbreak 

 Let $S_0$ denote the square with vertices $\pm 2 K \overline{\gamma}$ and $\pm 2 K \gamma$; let $S_1$ denote the square $S_0 + 2 L$ and $S_2$ the square $S_0 + 4 L$ obtained after shifting $S_0$ to the right by $2 L$ and $4 L$ respectively. Let $z_0$ be a point of $S_0$ (other than a vertex): let $z_1 \in S_1$ be the image of $z_0$ under reflexion over the edge $(2 K \overline{\gamma}, 2 K \gamma)$ shared by $S_0$ and $S_1$; let $z_2 \in S_2$ be the image of $z_1$ under reflexion over the edge $(6 K \overline{\gamma}, 6 K \gamma)$ shared by $S_1$ and $S_2$; and notice that $z_2 = z_0 + 4 L$. As $s$ is real-valued on $(2 K \overline{\gamma}, 2 K \gamma)$ it follows that $s(z_1) = \overline{s(z_0)}$; as the reflexion-extended $s$ is real-valued along $(6 K \overline{\gamma}, 6 K \gamma)$ it follows that $s(z_2) = \overline{s(z_1)}$. Thus 
$s(z_0 + 4 L) = s(z_2) = \overline{s(z_1)} = s(z_0)$
and so the extended $s$ has $4 L$ as a period. Similarly, reflexion in imaginary directions shows that $s$ has $4 \ii L$ as a period. 

\medbreak 

This application of the Schwarz Reflexion Principle extends $s$ to a function that is elliptic: it is doubly-periodic, with $4 L$ and $4 \ii L$ as periods (not fundamental - see Section 5); its singularities are simple poles at $\{ \pm 2 K \gamma, \, \pm 2 K \overline{\gamma} \}$ and points congruent modulo periods. This elliptic function is the full lemniscatic sine. 

\medbreak 

Naturally, the various analytic identities that were satisfied by $s$ in its ancestral versions hold also for the full meromorphic extension. For instance, $s$ is `real', odd, equivariant under multiplication by $\ii$, and continues to satisfy the duplication formula that opened the present section. 

\medbreak 

\section{The Lemniscatic Cosine} 

\medbreak 

We now direct further attention towards the lemniscatic cosine function $c$. Recall that in Section 2 we established the existence and elementary properties of $c$ in the open disc $B_r (0)$ of radius $r = 2^{-1/2}$ about $0$. The task of extending the holomorphic function $c$ in this disc to a meromorphic function $c$ in the plane may be accomplished in a variety of ways. We here outline several approaches because of their intrinsic interest, assigning some of the details as exercises; we leave the simplest approach for last, including all of the details as they are so few in number. Naturally, these various approaches lead to the same meromorphic function, by virtue of the Identity Theorem. 

\medbreak 

First of all, we may adapt for $c$ the approach that was taken to extending $s$: that is, we may extend $c$ from $B_r (0)$ by reduplication until the process of extension can be left to the Schwarz Reflexion Principle. For this purpose, we need a duplication formula for $c$ that involves $c$ only: one such formula is 
$$c(2 z) = - \: \frac{c(z)^4 + 2 c(z)^2 - 1}{c(z)^4 - 2 c(z)^2 - 1};$$ 
another is 
$$c(2 z) = \frac{2 c(z)^2 - c\,'(z)^2}{2 c(z)^2 + c\,'(z)^2}.$$
In this approach, poles of $c$ are encountered upon the very first duplication: the formula of Theorem \ref{Py} makes it clear that poles of $c$ coincide with points at which $s = \pm \ii$; the disc $B_{2 r} (0)$ contains two such points, namely $\pm \ii L$. The details of this approach are left as an exercise modelled on Section 3. 

\medbreak 

We may contemplate carrying this approach back to the beginning and attempt to develop $c$ from the initial value problem 
$$(c\,')^2 = 1 - c^4 ; \; \; c(0) = 1.$$ 
Unfortunately, this initial value problem does not have just one solution: along with the lemniscatic cosine, it has as a solution the function with constant value $1$; the Picard existence-uniqueness theorem does not apply, because the requisite Lipschitz condition is not satisfied. However, see the discussion of the simplest approach below. 

\medbreak 

Rather than follow alongside the path by which we extended the lemniscatic sine, we may instead take the extended lemniscatic sine and fashion from it the extended lemniscatic cosine. We proceed to consider three such approaches. 

\medbreak 

We may take a cue from the formula of Theorem \ref{Py}: with the fully extended $s$ in hand, we may define $c$ by starting from the formula 
$$c^2 = \frac{1 - s^2}{1 + s^2}$$
and then passing to a square-root; we specify the root by recalling the condition $c(0) = 1$. Notice that where $s$ has a pole, the quotient $(1 - s^2)/(1 + s^2)$ has a removable singularity with cured value $-1$ so that $c$ has value $\pm \ii$; this is as expected. One matter does call for serious attention: the very existence of a meromorphic square-root. This is easily settled: the meromorphic quotient $(1 - s^2)/(1 + s^2)$ has double zeros and double poles, as may be readily checked; as the zeros and poles have even orders, the quotient has a meromorphic square-root as a consequence of the Weierstrass Factorization Theorem. 

\medbreak 

We may take a cue from the proof of Theorem \ref{IVP}: with $s$ fully extended as above, we may adopt the formula 
$$c = \frac{s\,'}{1 + s^2}$$
as a definition of $c$. The obstacles encountered on this route are of largely cosmetic character. There are singularities both where the denominator is zero and where either numerator or denominator has a pole. The zeros of $1 + s^2$ are double and serve also as simple zeros of $s\,'$; accordingly, these points are simple poles for $s\,'/(1 + s^2)$. Poles of the numerator coincide with poles of the denominator, both having order two; accordingly, these points are removable singularities of $s\,'/(1 + s^2)$ and they have cured value $\pm \ii$ as they should. 

\medbreak 

Finally, the most transparent approach of all is simply to define $c$ by the rule 
$$c(z) = s (L - z)$$
for all $z \in \CC$ such that $L - z$ is not a pole of $s$. For clarity, let us temporarily write 
$$C(z) = s(L - z)$$
for such $z$. This plainly defines a meromorphic function $C$; all we need do is verify that it restricts to $B_r (0)$ as the original lemniscatic cosine $c$. Certainly $C$ satisfies the initial condition $C(0) = 1$ because $s(L) = 1$. Also, $C\,' (z) = - s\,' (L - z)$ so that 
$$C\,' (z)^2 = s\,'(L - z)^2 = 1 - s(L - z)^4 = 1 - C(z)^4.$$
Unfortunately, as mentioned above, the first-order initial value problem `$(c\,')^2 = 1 - c^4; \; c(0) = 1$' is inadequate for singling out $c$ in $B_r (0)$. Fortunately, the second-order initial value problem 
$$c\,'' = - 2 c^3; \; \; c(0) = 1, \; c\,'(0) = 0$$ 
noted after Theorem \ref{IVP} is adequate for this purpose. From above, in addition to $C(0) = s(L) = 1$ and $C\,'(0) = - s\,'(L) = 0$ we have 
$$C\,''(z) = s\,''(L - z) = - 2 s(L - z)^3 = - 2 C(z)^3$$
because $s\,'' = - 2 s^3$ as noted after Theorem \ref{s}. This is enough to ensure that the meromorphic function $C$ extends the holomorphic function $c : B_r (0) \to \CC$. As usual, we drop the capitalization and refer to $C$ as simply $c$; this is the full lemniscatic cosine.   

\medbreak 

The properties of this full lemniscatic cosine may be deduced immediately from those of the full lemniscatic sine: $c$ has $4 L$ and $4 \ii L$ as periods; its poles are at the points $\{ \pm \ii L, \, 2 L \pm \ii L \}$ and points congruent modulo periods. 

\medbreak 

Naturally, this meromorphic extension $c$ continues to be `real', to be even and to be reciprocated under multiplication by $\ii$; further, it satisfies the duplication formulae that were announced at the start of the present section. 

\medbreak 

\section{Remarks} 

\medbreak 

In this closing section, we gather a number of miscellaneous observations regarding the lemniscatic functions, leaving some of the details as exercises. 

\medbreak 

As we have seen, the `Pythagorean' identity 
$$s^2 + s^2 c^2 + c^2 = 1$$
in the form 
$$(1 + s^2)(1 + c^2) = 2.$$
has consequences for $s$ and $c$: thus, either of these functions has poles exactly where the other has value $\pm \ii$; also, either function has zeros exactly where the other has value $\pm 1$. 

\medbreak 

The fundamental `complementary' relationship between $s$ and $c$ expressed in the formula 
$$c(z) = s(L - z)$$
has its own consequences. For example, in conjunction with Theorem \ref{i} it yields 
$$s(L - \ii z) = c(\ii z) = \frac{1}{c(z)} = \frac{1}{s(L - z)}$$
from which we deduce that $L - z$ is a zero of $s$ precisely when $L - \ii z$ is a pole of $s$. It follows from this and Theorem \ref{i} that the zero-set $Z_s$ of $s$ is related to its pole-set $P_s$ by 
$$Z_s = P_s \pm (L + \ii L).$$ 
The pole-set $P_c$ of $c$ is more directly related to its zero-set $Z_c$: indeed, Theorem \ref{i} shows at once that 
$$Z_c = \ii P_c.$$	

\medbreak 

Let us return to a consideration of the square with vertices $\pm 2 K \overline{\gamma}$ and $\pm 2 K \gamma$. Theorem \ref{i} makes clear the behaviour of $c$ on the diagonals of this square: if $z$ lies on one of these diagonals then $\overline{z} = \pm \ii z$ and therefore 
$$\overline{c(z)} = c(\overline{z}) = c(\pm \ii z) = 1/c(z);$$ 
thus $c(z)$ lies on the unit circle. The `complementary' identity of the preceding paragraph permits us to deduce from this that the values of $s$ around the perimeter of the square with vertices $\{ \pm L, \, \pm \ii L \}$ also lie on the unit circle. 

\medbreak 

Incidentally, recall that we extended $s$ from the square with vertices $\pm 2 K \overline{\gamma}$ and $\pm 2 K \gamma$ by means of the Schwarz Reflexion Principle. As the values of $s$ along the interval $(0, 2 K \gamma)$ lie in the line $\gamma \R$, the Schwarz Reflexion Principle enables us to recover $s$ on this square from $s$ on the triangle with vertices $\{ 0, L, 2 K \gamma \}$. We can go further: since a version of the Schwarz Reflexion Principle applies across circular arcs, we may in fact recover $s$ from its restriction to the triangle with vertices $\{ 0, L, K \gamma \}$; by the same token, we may instead generate the full $s$ from its restriction to the square with vertices $\{ \pm L, \, \pm \ii L \}.$

\medbreak 

Before passing on to other topics, we pause to record the following elementary consequences of the same `complementary' identity. As $c$ is even, 
$$s(L + z) = s(L - (-z)) = c(-z) = c(z);$$
as $s$ is odd, 
$$c(L + z) = c(L - (-z)) = s(-z) = -s(z).$$ 
\medbreak 
\noindent
Consequently, 
$$s(2 L + z) = - s(z) \; \; {\rm and} \; \; c(2 L + z) = - c(z)$$
whence we recover $4 L$ as a period of both $s$ and $c$. Similarly, 
$$s(2 \ii L + z) = - s(z) \; \; {\rm and} \; \; c(2 \ii L + z) = - c(z)$$
whence we recover $4 \ii L$ as a period of $s$ and $c$. 

\medbreak 

As elliptic functions, $s$ and $c$ have not only duplication formulae but also addition formulae, which assume a variety of shapes. One version of the addition formula for $s$ alone reads 
$$s(a + z) = \frac{s\,'(a) \, s(z) + s(a) \, s\,'(z)}{1 + s(a)^2 \, s(z)^2}$$
where $a$ and $z$ are such that both sides make sense. This may first be verified for constant $a \in B_r (0)$ and variable $z \in B_r(0)$ where everything is holomorphic: after some calculation, it is found that both sides satisfy the same initial value problem at $z = 0$; they therefore coincide by the Picard theorem and the Identity Theorem. Alternatively, $s$ and $c$ have joint addition formulae: 
$$s(a + z) = \frac{s(a) c(z) + c(a) s(z)}{1 - s(a) c(a) s(z) c(z)}$$
and 
$$c(a + z) = \frac{c(a) c(z) - s(a) s(z)}{1 + s(a) c(a) s(z) c(z)}\, .$$
\medbreak 

\medbreak 

As coperiodic elliptic functions, $s$ and $c$ share a coperiodic Weierstrass function. Before we identify this Weierstrass function, we should recognize that the periods $\{ \pm 4 L, \, \pm 4 \ii L \}$ do not contain a fundamental set for $s$ and $c$. Recall from above that the addition of either $2 L$ or $2 \ii L$ to the argument of either function reverses the value of the function. It follows from this that $ \pm 2 L \pm 2 i L$ is a period for any combination of signs. Any two of these constitute a fundamental set: if we shift the parallelogram with vertices $\{ 0, 2 L + 2 \ii L, 2 L - 2 \ii L, 4 L \}$ a little to the left, then the shifted parallelogram surrounds only the two simple poles $ 2 K \gamma$ and $2 K \overline{\gamma}$ of $s$. 

\medbreak 

The Weierstrass function $\wp$ coperiodic with the lemniscatic functions $s$ and $c$ is given by the rule 
$$\wp(z) = \tfrac{1}{2} \, \ii\, s(\tfrac{1}{2}(z + \ii z))^{-2}$$
\medbreak 
\noindent 
as is readily verified; otherwise said, with $\lambda = \gamma / \sqrt2 = \tfrac{1}{2} (1 + \ii)$, 
$$\wp(z) = \frac{\lambda^2}{s(\lambda z)^2}\, .$$
Note that 
$$\wp\,' (z) = - 2 \lambda^3 s(\lambda z)^{-3} s\,'(\lambda z)$$
so that 
$$\wp\,'(z)^2 = 4 \lambda^6 s(\lambda z)^{-6} (1 - s( \lambda z)^4)$$
and therefore 
$$\wp\,'(z)^2 = 4 \wp(z)^3 + \wp(z).$$
Thus, the Weierstrass function $\wp$ has invariants $g_2 = -1$ and $g_3 = 0$; it is {\it pseudolemniscatic}. 

\medbreak 

The pseudolemniscatic period lattice of $s$ and $c$ is not Jacobian, hence no true Jacobian functions ${\rm sn}$, ${\rm cn}$ and ${\rm dn}$ are associated to it. However, Jacobian functions have counterparts for any Weierstrass function: namely, a triple of functions called `primitive' by Neville in his classic account [3]. We proceed to relate the lemniscatic functions $s$ and $c$ to the `primitive' functions associated to the pseudolemniscatic Weierstrass function $\wp$. In what follows, we use freely the notation and terminology of [3]. 

\medbreak 

As half-periods, we take the triple 
$$\omega_f = 2 L, \; \omega_g = - L + \ii L \; \; {\rm and}\;  \; \omega_h = -L - \ii L$$ 
of which any pair is fundamental; the corresponding midpoint values of $\wp$ are found to be 
$$e_f = \wp(\omega_f) = 0, \; e_g = \wp(\omega_g) = \tfrac{1}{2} \ii \; \; {\rm and} \; \; e_h = \wp(\omega_h) = - \tfrac{1}{2} \ii.$$

\medbreak 

The `primitive' function ${\rm fj}$ of Neville is the meromorphic function defined by the rule  
$${\rm fj} \, (z)^2  = \wp(z) - e_f$$
and the requirement 
$$z \, {\rm fj}\,(z) \to 1 \; {\rm as} \; z \to 0.$$ 
Here, $e_f = 0$ and the function ${\rm fj}$ may be read directly from the explicit formula for $\wp$ given above: thus, 
$${\rm fj} \, (z) = \lambda \, s(\lambda \, z)^{-1} = \tfrac{1}{2}(1 + \ii) s(\tfrac{1}{2}(z + \ii z))^{- 1}.$$

\medbreak 

The analogously-defined functions ${\rm gj}$ and ${\rm hj}$ lie a little deeper. Passage to the square-root in 
$${\rm gj} \, (z)^2  = \wp(z) - \tfrac{1}{2} \ii$$
and
$${\rm hj} \, (z)^2  = \wp(z) + \tfrac{1}{2} \ii$$ 
is facilitated by the fact that the function $1 + s^2$ has a canonical square-root: explicitly, it may be checked that 
$$ 1 + s(2 w)^2 = \Big( \frac{1 + s(w)^2 c(w)^2}{1 - s(w)^2 c(w)^2} \Big)^2.$$ 
\medbreak 
\noindent 
Some calculation, again with $\lambda = \tfrac{1}{2} (1 + \ii)$ for convenience, results in the explicit formulae 
$${\rm gj} \, (2 z) = \lambda  \, \frac{c(\lambda z)^2  - s(\lambda z)^2}{2 s(\lambda z) \, c(\lambda z)}$$
and 
$${\rm hj} \, (2 z) = \lambda \, \frac{1 + s(\lambda z)^2 \, c(\lambda z)^2}{2 s(\lambda z) \, c(\lambda z)}\, .$$

\medbreak 

In the opposite direction, we may recover $s$ and $c$ from the `primitive' functions: thus, it may be verified that 
$$s(z) = \lambda \, {\rm fj} \, (z / \lambda)^{-1}$$
and 
$$c(z) = \frac{{\rm gj} \, (z / \lambda)}{{\rm hj} \, (z / \lambda)}\, .$$
\medbreak 
\noindent 
Alternatively, we may express $s$ in terms of the `elementary' function ${\rm jf}$ and $c$ in terms of the `elementary' function ${\rm gh}$; we refer to [3] for the definitions, merely recording the formulae 
$$s(z) = - 2 \lambda \, {\rm jf} \, (z / \lambda) = - (1 + \ii) {\rm jf} \, (z - \ii z)$$
and 
$$c(z) = 2 \lambda \, {\rm gh} \, (z / \lambda) = (1 + \ii) \, {\rm gh} \, (z - \ii z)\, .$$
\medbreak 

\medbreak 

We should also mention the Jacobian approach to the lemniscatic functions, to which we alluded at the start of this paper. For this purpose, we take the Weierstrass function $P$ that has $4 K$ and $4 \ii K$ as a fundamental pair of periods: this is given by the explicit formula 
$$P(z) = \tfrac{1}{2} \, s(\tfrac{1}{\sqrt2} z)^{-2}.$$
By direct calculation, 
$$P\,'(z)^2 = 4 P(z)^3 - P(z)$$
so that $P$ has invariants $g_2 = 1$ and $g_3 = 0$; this Weierstrass function is {\it lemniscatic} and its period lattice is truly Jacobian. The corresponding Jacobian functions ${\rm sn}$, ${\rm cn}$ and ${\rm dn}$ have self-complementary modulus $1/ \sqrt{2}$; in terms of the Jacobian function ${\rm cn}$ and the Glaisher quotient ${\rm sd} = {\rm sn} / {\rm dn}$, the lemniscatic functions $s$ and $c$ are given by 
$$s(z) = \tfrac{1}{\sqrt2} \, {\rm sd} \, (\sqrt2 z)$$
and 
$$c(z) = {\rm cn} \, (\sqrt2 z).$$ 

\medbreak 

It would be remiss of us not to mention that the numbers $K$ and $L = \sqrt2 K$ arising from our analysis already have names: in fact, it is more-or-less customary to write 
$$\int_0^1 \frac{{\rm d} \sigma \; \; \; }{(1 - \sigma^4)^{1/2}} =  \frac{1}{2} \, \varpi$$
by analogy with 
$$\int_0^1 \frac{{\rm d} \sigma \; \; \; }{(1 - \sigma^2)^{1/2}} =  \frac{1}{2} \, \pi;$$
the analogy is supported by the `complementary' relationship $c(z) = s(\frac{1}{2} \, \varpi - z)$ and the fact that $s$ and $c$ have $2 \varpi$ as a period. 

\medbreak 

Finally, a few references are in order. We recommend Chapter 2 of [1] for the Picard existence-uniqueness theorem along with much related material. For a fuller account of the use of Schwarz reflexions for constructing elliptic functions, see Section 3 of Chapter VI in [2].  The masterly treatment of Jacobian elliptic functions in [3] also includes an instructive Weierstrassian use of duplication to extend the domain of a meromorphic function. An enlightening account of lemniscatic integrals in connexion with the problem of doubling lemniscatic arcs may be found in [4].  Lemniscatic functions are presented as special cases of Jacobian elliptic functions in Chapter XXII of the classic treatise [5].

\bigbreak

\begin{center} 
{\small R}{\footnotesize EFERENCES}
\end{center} 
\medbreak 

[1] E. Hille, {\it Ordinary Differential Equations in the Complex Domain}, Wiley-Interscience (1976); Dover Publications (1997).

\medbreak 

[2] Z. Nehari, {\it Conformal Mapping}, McGraw-Hill (1952); Dover Publications (1975). 

\medbreak 

[3] E.H. Neville, {\it Jacobian Elliptic Functions}, Oxford University Press (1944). 

\medbreak 

[4] C.L. Siegel, {\it Topics in Complex Function Theory}, Volume I, Wiley-Interscience (1969). 

\medbreak 

[5] E. T. Whittaker and G. N. Watson, {\it A Course of Modern Analysis}, Second Edition, Cambridge University Press (1915).

\medbreak

\end{document}